\documentclass[12pt, onecolumn,draftcls]{IEEEtran} 
\usepackage{amssymb,amsmath}
\usepackage{graphicx}
\usepackage{color}
\newtheorem{thm}{Theorem}[section]
\newtheorem{prop}[thm]{Proposition}
\newtheorem{lem}[thm]{Lemma}
\newtheorem{rem}{Remark}[section]

\begin{document}

\title{Snapshot-based Balanced Truncation for Linear Time-periodic Systems}

\author{Zhanhua Ma, \and Clarence W. Rowley,\\{\small Department of Mechanical and Aerospace Engineering, Princeton University} \\\and  Gilead Tadmor\\{\small Department of Electrical and Computer Engineering, Northeastern University} \thanks{Z.~Ma and C.~W.~Rowley: MAE Department, Engineering Quad, Princeton University, Princeton, NJ 08544 (E-mail: zma@princeton.edu; cwrowley@princeton.edu; Phone: (609) 258-7321; Fax: (609) 258-6109). }\thanks{G.~Tadmor: Department of Electrical and Computer Engineering, Northeastern University, Boston, MA 02115 (E-mail: tadmor@ece.neu.edu; Phone: (617) 373-5277;
Fax: (617) 373-4189).}\thanks{Corresponding author: C. W. Rowley. }\thanks{This work was supported by AFOSR (grant FA9550-05-1-0369) and NSF (grant CMS-0347239).}} 
\maketitle
\begin{abstract}
We introduce an algorithm based on a method of snapshots for computing approximate balanced truncations for discrete-time, stable, linear time-periodic systems. By construction, this algorithm is applicable to very high-dimensional systems, even with very high-dimensional outputs (or, alternatively, very high-dimensional inputs).  An example is shown to validate the method.
\end{abstract}

\section{Introduction}
We consider the model reduction problem for stable, discrete-time periodic systems, using an approximation of balanced trunctation.  In particular, we combine the lifting approach developed in~\cite{FaBeDu2005} with a snapshot-based approximation described in~\cite{Ro2005}, which is tractable even for very high-dimensional systems, on the order of millions of states.

Several different algorithms are available for extending balanced truncation from linear time-invariant systems to time-periodic systems (see, e.g., \cite{Mo1981,LaBeDu1998,FaBeDu2005}).  Here, our interest is in systems with very large state dimension, on the order of tens of thousands or millions, for which the algorithms used in these previous approaches become intractable.  Such systems arise in the control of systems governed by partial differential equations, for instance in fluid mechanics, where often one approximates the full infinite-dimensional dynamics by a high-dimensional discretization.  Periodic systems may arise in these settings either as linearizations around a periodic orbit (e.g., vortex shedding~\cite{Noack-03}), or as linearizations of a system with periodic open-loop forcing (e.g., periodic pulsed blowing).  The goal of this paper is to describe a model reduction procedure that closely approximates balanced truncation, yet is computationally tractable even for these high-dimensional systems.  The resulting reduced-order models may be used for control synthesis, or other studies where the full high-dimensional system is unwieldy.

We suppose that while the number of states and outputs are both very large, the number of control inputs is small (by duality, we may alternatively assume that the number of outputs is small).  This case is also typical in practice, in which one often has a small number of actuators, or sources of disturbances, that one wants to model.  We also assume that the system is asymptotically stable.  Key ideas in the method presented here involve the computation of the balancing transformation directly from snapshots, without computing the controllability and observability Gramians, which become prohibitively large for these systems; and two different output projection methods for tractable computation with systems with high-dimensional output spaces.

The paper is organized as follows: in Section~\ref{background_sec} we summarize an existing approach to balanced truncation for time-periodic systems, and the corresponding lifted time-invariant  reformulation. In Section~\ref{bpod_sec} we present the main results, an approximate method based on snapshots, and an output projection based on the lifting approach that makes the snapshot method feasible for large numbers of outputs.  In Section~\ref{example_sec}, we demonstrate the method on a numerical example.

\section{Background on balanced truncation and linear periodic systems}
\label{background_sec}
\subsection{Balanced truncations for linear time-invariant systems}
Consider a discrete linear time-periodic system
\begin{equation}
	\begin{aligned}
	x(k+1)&=A(k)x(k)+B(k)u(k)\\
	y(k)&=C(k)x(k),
	\end{aligned}	
	\label{LPS}
\end{equation}
with state variable $x(k)\in\mathbb{C}^n$, input $u(k)\in\mathbb{C}^p$ and output $y\in\mathbb{C}^q$, and $k\in\mathbb{Z}$, where the matrices $A(k)\in\mathbb{C}^{n\times n}$, $B(k)\in\mathbb{C}^{n\times p}$ and $C(k)\in\mathbb{C}^{q\times n}$ are $T$-periodic (e.g., $A(k+T)=A(k)$ for all~$k$).

We assume that this system is asymptotically stable, in the following sense: defining $F(j,i)=A_{j-1}A_{j-2}\cdots A_i$ for $j>i$, with $F(i,i)=I_{n\times n}$, the $n\times n$ identity matrix, we consider the case where the spectral radius~$\rho\left(F(j+T, j)\right)$ is smaller than~$1$.  Note that periodicity implies that the non-zero eigenvalues of $F(j+T, j)$ are independent of time~$j$ \cite{GraLon1991}, whereby asymptotic stability is equivalent to uniform geometric decay, as in the LTI case.

Many different approaches are available for model reduction, including singular perturbation~\cite{Kokotovic:1999}, Hankel norm reduction, and balanced truncation.  Balanced truncation~\cite{Mo1981} has become a widely used method for linear systems, since it has a priori error bounds comparable to other methods, and is computationally tractable, at least for systems of moderate dimension, say $n<10^4$.  However, for very large systems, exact balanced truncation becomes computationally intractable, as the procedure involves finding coordinate transformations that simultaneously diagonalize non-sparse $n\times n$ matrices (controllability and observability Gramians).  To overcome this difficulty, a snapshot-based method has been proposed for approximate balanced truncation~\cite{Ro2005}, sometimes referred to as {\em balanced proper orthogonal decomposition} (balanced POD).  Here, the idea is to use ``empirical Gramians'' defined using snapshots of a simulation of the system, as in~\cite{LallMarsden-02}, and to compute the balancing transformation directly from the snapshots, without ever computing the Gramians.

The method of balanced truncation has been applied to linear time-periodic systems using several different approaches. One approach, used in~\cite{LaBeDu1998,LoOr1999,Varga2000,SaRa2004}, is to perform $T$ separate balanced truncations along a period. An alternative, suggested by~\cite{FaBeDu2005}, is a lifting approach, in which balanced truncation is applied to a time-invariant reformulation.   However, both of these approaches require the solution of Lyapunov equations or inequalities and are not computationally tractable for very large systems.  Thus, the objective of our  work is to obtain more computationally efficient algorithms for computing approximate balanced truncations for these large systems.

Before reviewing the theory for periodic systems, first recall the main idea of balanced truncation for linear time-invariant (LTI) stable systems (for more detail, see standard textbooks, such as~\cite{ZhDoGl1996,DuPa1999}).  One begins by defining controllability and observability Gramians $W_c$ and $W_o$ to measure to what degree each state is excited by an input, and each state excites future outputs, respectively.  One then seeks a coordinate transformation $\Phi$, called balancing transformation, so that the transformed Gramians $W_c\mapsto \Phi^{-1}W_c(\Phi^{-1})^*$ and $W_o\mapsto \Phi^*W_o\Phi$ are equal and diagonal (hence balanced): $\Phi^{-1}W_c(\Phi^{-1})^*=\Phi^*W_o\Phi=\Sigma=\operatorname{diag}(\sigma_1,\cdots,\sigma_n)$, where $\sigma_1\geqslant\cdots\geqslant\sigma_n\geqslant 0$ are the Hankel singular values. Finally, in these new coordinates, one truncates the states that are least controllable and observable, corresponding to the smallest Hankel singular values.

\subsection{Controllability and Observability Gramians for linear periodic systems}
For time-periodic systems, one may also define controllability and observability Gramians, as in~\cite{Varga2000}.   The \emph{controllability Gramian at time $j$}, $W_c(j)$, and \emph{observability Gramian at time $j$}, $W_o(j)$, for the complex stable periodic system (\ref{LPS}) are defined similarly by the following two positive-semidefinite $n\times n$ matrices
\begin{equation}
\begin{aligned}
W_c(j)&:=\sum_{i=-\infty}^{j-1}F(j,i+1)B(i)B(i)^*F(j,i+1)^*\\
W_o(j)&:=\sum_{i=j}^{\infty}F(i,j)^*C(i)^*C(i)F(i,j),
\end{aligned}
\label{Grammian_def}
\end{equation}
where the superscript $^*$ denotes the adjoint of a linear operator.  See the Appendix for the specific properties these Gramians satisfy.  Both $W_c(j)$ and $W_o(j)$ are $T$-periodic.    Also, for each $j$, $W_c(j)$ and $W_o(j)$ satisfy the following respective discrete Lyapunov equations:
\begin{equation}
\begin{aligned}
A(j)W_c(j)A(j)^*-W_c(j+1)+B(j)B(j)^*&=0\\
A(j)^*W_o(j+1)A(j)-W_o(j)+C(j)^*C(j)&=0.
\end{aligned}
\label{LyapunovEqn}
\end{equation}
As discussed in the Appendix, when convergent, the definitions (\ref{Grammian_def}) (equivalently,~(\ref{LyapunovEqn})) extend to general linear time-varying systems of the form~(\ref{LPS}), with the same interpretations in terms of input-to-state and state-to-output mappings.

\subsection{Traditional approach of balanced truncation for periodic systems}
If the dimension~$n$ of the system is not too large, it is possible to solve the above pairs of Lyapunov equations for Gramians $W_c(j)$, $W_o(j)$ for each $j$, $1\leqslant j\leqslant T$.  One can then do balanced truncation \emph{at each time $j$}.   Refer to \cite{LaBeDu1998}, \cite{LoOr1999}, \cite{SaRa2004} for detailed discussions.  Note that Farhood et al. \cite{FaBeDu2005} also compute the Gramians/generalized Gramians by solving corresponding Lyapunov equations/inequalities at each time step; however, the balanced truncation is then realized in a lifted time-invariant setting, not in the original periodic one.  In this paper we will also use the lifting approach, in conjunction with a method of snapshots as in \cite{Ro2005}.

\subsection{The lifted system}
As discussed in~\cite{GoKaLe1992,FaBeDu2005}, the linear periodic time-variant system (\ref{LPS}) can be equivalently rewritten as a linear time invariant (LTI) system:
\begin{equation}
\begin{aligned}
\tilde{x}_j(t+1)&=\tilde{A}_j\tilde{x}_j(t)+\tilde{B}_j\tilde{u}_j(t)\\
\tilde{y}_j(t)&=\tilde{C}_j\tilde{x}_j(t)+\tilde{D}_j\tilde{u}_j(t),
\end{aligned}
\label{LTI_sys}
\end{equation}
where $j$ is fixed and 
\begin{align*}
\tilde{x}_j(t)&=x(j+tT)\\
\tilde{y}_j(t)&=
\begin{bmatrix}
y(j+tT)\\y(j+tT+1)\\\vdots\\y\left(j+(t+1)T-1\right)
\end{bmatrix}\\
\tilde{u}_j(t)&=\begin{bmatrix}
u(j+tT)\\u(j+tT+1)\\\vdots\\u\left(j+(t+1)T-1\right)
\end{bmatrix}
\end{align*}
and
\begin{align*}
\tilde{A}_j&=F(j+T,j)=A(j+T-1)A(j+T-2)\cdots A(j)\\
\tilde{B}_j&=\big[E (j+T,j+1)B(j),\cdots,\\
	&\qquad E (j+T,j+T-1)B(j+T-2), B(j+T-1)\big]\\
\tilde{C}_j&=\begin{bmatrix}
C(j)\\C(j+1)F(j+1,j)\\\vdots\\C(j+T-1)F(j+T-1,j)
\end{bmatrix}\\
\tilde{D}_j&=\begin{bmatrix}
0& & & 0\\F_{j,2,1}&0& & \\\vdots &\vdots & \ddots & \\F_{j,T,1} & F_{j,T,2} & \cdots & 0
\end{bmatrix},
\end{align*}
where $F_{j,i,k}=C(j+i-1)F(j+i-1,j+k)B(j+k-1)$.  We refer to the time invariant system (\ref{LTI_sys}) as a \emph{lifted system at time $j$}, in which $\tilde{A}_j$, $\tilde{B}_j$, $\tilde{C}_j$ and $\tilde{D}_j$ are constant matrices.  The lifted system may be viewed as a Poincar\'{e} map of the original periodic system in the state space at times $(j+tT)$, while the input and output information for an entire period is kept at each iteration of the Poincar\'e map.

The controllability Gramian $\tilde{W}_{jc}$ and observability Gramian $\tilde{W}_{jo}$ for this LTI lifted system are conventionally defined by
the two positive-semidefinite $n\times n$ matrices 
\begin{equation}
\begin{aligned}
\tilde{W}_{jc}&:=\sum_{i=0}^{\infty}\tilde{A}_j^i\tilde{B}_j\tilde{B}_j^*\left(\tilde{A}_j^i\right)^*\\
\tilde{W}_{jo}&:=\sum_{i=0}^{\infty}\left(\tilde{A}_j^i\right)^*\tilde{C}_j^*\tilde{C}_j\tilde{A}_j^i.
\end{aligned}
\label{Grammian_def_LTI}
\end{equation}
They respectively satisfy the following discrete Lyapunov equations:
\begin{equation}
\begin{aligned}
\tilde{A}_j\tilde{W}_{jc}\tilde{A}_j^*-\tilde{W}_{jc}+\tilde{B}_{j}\tilde{B}_{j}^*&=0\\
\tilde{A}_j^*\tilde{W}_{jo}\tilde{A}_j-\tilde{W}_{jo}+\tilde{C}_{j}^*\tilde{C}_{j}&=0.
\end{aligned}
\label{LyapunovEqn_LTI}
\end{equation}
The above setting allows one to study the lifted LTI system instead of the original periodic one such that the well-developed balanced truncation techniques for LTI systems are available.  We emphasize that in the lifted LTI system,  only one time of balanced truncation is needed.    However, the drawbacks of working on the lifted system are that it has even higher dimensional inputs and outputs, and that we usually do not have the explicit   form of the lifted system in computations.  The approach taken in this paper is as follows: \emph{by using the equivalence between the LTI (lifted) and time-periodic systems, we apply the balanced POD procedure in the LTI setting but do computations in the original periodic setting}.  To realize this plan, first we list the following statement establishing the equivalent relations between the Gramians of the lifted system and the  original periodic system. Proofs are by direct calculations, using periodicity. 
\begin{prop}
 Consider an arbitrary time $j$.  The controllability and observability Gramians at time $j$ of the periodic system (\ref{LPS}), i.e., $W_c(j)$ and $W_o(j)$ given by (\ref{Grammian_def})   are respectively equal to the controllability and observability Gramians $\tilde{W}_{jc}$ and $\tilde{W}_{jo}$ of the lifted system (\ref{LTI_sys}) at time $j$ given by (\ref{Grammian_def_LTI}).

\label{Prop Gramian Contrl Observ}
\end{prop}

An immediate corollary from the above statement is  that, as stated in~\cite{FaBeDu2005},~\cite{LoOr1999},  solutions of the time-varying Lyapunov equations~(\ref{LyapunovEqn}) are equivalent to the solution of the time-invariant Lyapunov equations~(\ref{LyapunovEqn_LTI}) in the lifted setting.  This result plays an important role in traditional lifting approaches of balanced truncations for periodic systems, where the Gramians are found by solving Lyapunov equations.  In the following section, we show how Proposition~\ref{Prop Gramian Contrl Observ} allows us to apply a snapshot-based algorithm for computing balancing transformations for periodic systems, thereby avoiding the need to solve Lyapunov equations.

\section{Balanced POD for linear periodic systems}
\label{bpod_sec}

\subsection{Method of snapshots for computation of empirical Gramians for time-varying discrete systems}
\label{snapshot_subsec}
The method of snapshots provides a numerical approximation of Gramians directly based on their definitions.  Refer to \cite{Ro2005} and the references therein for the application of this method to linear time-invariant systems.    

For time-varying discrete systems, first, we define the \emph{empirical controllability Gramian} and \emph{empirical observability Gramian} at time $j$ for the periodic system (\ref{LPS}) as 
\begin{align}
W_{ce}(j; m)&:=\sum_{i=j-m}^{j-1}F(j,i+1)B(i)B(i)^*F(j,i+1)^*
\label{Gramian_Controllability_Empirical}\\
W_{oe}(j; m)&:=\sum_{i=j}^{j+m-1}F(i,j)^*C(i)^*C(i)F(i,j),
\label{Gramian_Observability_Empirical}
\end{align}
where $m$ is a nonnegative integer. Clearly, $\lim_{m\rightarrow \infty}W_{ce}(j;m) = W_c(j)$; $\lim_{m\rightarrow \infty}W_{oe}(j;m) = W_o(j)$.  Further, the following result gives upper error bounds of $\|W_c(j)-W_{ce}(j;m)\|$ and $\|W_o(j)-W_{oe}(j;m)\|$.
\begin{lem}
Consider an arbitrary time $j$.  Let $m=lT$, where $l$ is a nonnegative integer.  For the linear periodic system (\ref{LPS}),
\begin{equation}
\begin{aligned}
\frac{\|W_c(j)-W_c(j;m)\|}{\|W_c(j)\|}&\leqslant \|F(j+T,j)^l\|^{2}\\ \frac{\|W_o(j)-W_o(j;m)\|}{\|W_o(j)\|}&\leqslant \|F(j+T,j)^l\|^{2},
\end{aligned}
\end{equation}
where an induced norm is used. 
\label{lemma Empirical Gramian error bound LPS}
\end{lem}

The proof is based on the following result for linear time-invariant systems.
\begin{lem}
Consider a linear stable time-invariant system
\begin{align*}
x(k+1) &= Ax(k)+Bu(k)\\
y(k)&=Cx(k),
\end{align*}
where $A$, $B$ and $C$ are constant matrices and whose controllability and observability Gramians are $W_c=\sum_{i=0}^{\infty}A^iBB^*\left(A^i\right)^*$ and $W_o=\sum_{i=0}^{\infty}\left(A^i\right)^*C^*CA^i$. Define the empirical controllability and observability Gramians for this LTI system as 
\begin{align}
W_{ce}(l)&:=\sum_{i=0}^{l-1}A^iBB^*\left(A^i\right)^*;
\label{Gramian_Controllability_Empirical LTI}\\
W_{oe}(l)&:=\sum_{i=0}^{l-1}\left(A^i\right)^*C^*CA^i,
\label{Gramian_Observability_Empirical LTI}
\end{align}
where $l$ is a nonnegative integer. Thus,
\begin{equation*}
\frac{\|W_c-W_{ce}(l)\|}{\|W_c\|}\leqslant \|A^l\|^{2},\qquad \frac{\|W_o-W_{oe}(l)\|}{\|W_o\|}\leqslant \|A^l\|^{2},
\end{equation*}
where an induced norm is used.

\label{lemma Empirical Gramian error bound LTI}
\end{lem}

\begin{proof}
Under any induced norm,
\begin{align*}
\|W_c-W_{ce}(l)\|&=\left\|\sum_{i=l}^{\infty}A^iBB^*\left(A^i\right)^*\right\|
=\left\|A^lW_c\left(A^l\right)^*\right\|\\
&\leqslant\|A^l\|\|W_c\|\left\|\left(A^l\right)^*\right\|=\|A^l\|^2\|W_c\|.
\end{align*}
Similarly, we have $\|W_o-W_{oe}(l)\|\leqslant\|A^l\|^{2}\|W_o\|$.
\end{proof}

\begin{proof}[Proof of Lemma~\ref{lemma Empirical Gramian error bound LPS}]
By Proposition \ref{Prop Gramian Contrl Observ}, $W_c(j)=\tilde{W}_{jc}$. Also, it is straightforward to show $W_{ce}(j;m)=\tilde{W}_{jce}(l)$, where $\tilde{W}_{jce}(l)$ is the empirical controllability Gramian of the lifted LTI system at time $j$, defined in form of (\ref{Gramian_Controllability_Empirical LTI}).  Thus, 
\begin{align*}
W_c(j)-W_{ce}(j;m) = \tilde{W}_{jc} - \tilde{W}_{jce}(l),
\end{align*}
and the result follows by applying Lemma \ref{lemma Empirical Gramian error bound LTI} to the lifted LTI system at time $j$. The result for the observability Gramian follows from a similar argument.
\end{proof}
Indeed, for any $\epsilon>0$, there is an induced norm $\|\cdot\|$ such that $\|F((j+T,j)\|\leqslant\rho\left(F(j+T,j)\right)+\epsilon$ (by Lemma 5.6.10 in~\cite{HoJo1985}). Thus, with an $\epsilon$ satisfying $\rho\left(F(j+T,j)\right)+\epsilon<1$, Lemma \ref{lemma Empirical Gramian error bound LPS} implies that under a certain matrix norm, 
\begin{align*}
\frac{\|W_c(j)-W_c(j;m)\|}{\|W_c(j)\|}&< \left(\rho\left(F(j+T,j)\right)+\epsilon\right)^{2l}\\ \frac{\|W_o(j)-W_o(j;m)\|}{\|W_o(j)\|}&<\left(\rho\left(F(j+T,j)\right)+\epsilon\right)^{2l},
\end{align*}
where $m=lT$.  Thus, a large enough $m$ guarantees that the error between empirical and exact Gramians is small.   This result encourages one to use empirical Gramians instead of exact ones to realize approximate balanced truncations, whenever it is difficult to compute exact Gramians. The following subsections introduce the method of snapshots by which we can calculate those empirical Gramians for linear periodic systems.  Also note that it is clear that the condition $m=lT$ is not crucial to obtain a result similar to that in the Lemma \ref{lemma Empirical Gramian error bound LPS}.

\subsubsection{Computation of the empirical controllability Gramian}
\label{sec:ctrb_gram}

Without loss of generality, let $1\leqslant j\leqslant T$.  The empirical controllability Gramian $W_{ce}(j;m)$ can be obtained by running a series of impulse-response simulations for system~(\ref{LPS}) as follows.  

Consider the $d$-th column of $B(k)$, denoted by $[B(k)]^d$. To obtain a finite sum $\sum_{i=j-m}^{j-1}F(j,i+1)[B(k)]^d\left([B(k)]^d\right)^*F(j,i+1)^*$, one starts the first simulation with the initial condition $x(j-m)=0$ and a unit impulse $u(j-m)=[0,\cdots,0,1,0,\cdots,0]^{\top}$ whose $d$-th entry is $1$.  By running $m$ steps, we obtain
\begin{align*}
x(j-m+1)&=[B(j-m)]^d;\\
x(j-m+2)&=F(j-m+2,j-m+1)[B(j-m)]^d;\\
&\;\vdots\\
x(j)&=F(j,j-m+1)[B(j-m)]^d.
\end{align*}
Notice that $x(j)$ is a snapshot we need since $x(j)x(j)^*$ appears in the finite sum.  Similarly, we run the second simulation with initial condition $x(j-m+1)=0$ and unit impulse $u(j-m+1)=[0,\cdots,0,1,0,\cdots,0]^{\top}$, and by running $(m-1)$ steps we get $x(j-m+2)=[B(j-m+1)]^i$, $\cdots$, $x(j)=F(j,j-m+2)[B(j-m+1)]^i$, where the $x(j)$ will be used to construct the finite sum.  Repeat the above process until the $m$-th simulation, with initial condition $x(j-1)=0$ and unit impulse $u(j-1)=[0,\cdots,0,1,0,\cdots,0]^{\top}$, where only one step of simulation is needed to obtain $x(j)=[B(j-1)]^d$.  By running the $m$ simulations with $m+\cdots+1=m(m+1)/2$ steps in all, we can thus construct the sum $\sum_{i=j-m}^{j-1}F(j,i+1)[B(k)]^d\left([B(k)]^d\right)^*F(j,i+1)^*$.  Since $B(k)$ has $p$ columns, we need to run $mp$ simulations, with $m(m+1)p/2$ steps total, to obtain $mp$ snapshots.  Clearly, we have
\begin{prop}{\bf (Controllability Gramian from snapshots)}
Define an $n\times mp$ dimensional matrix
\begin{equation}
    \begin{gathered}
    X(j;m)=\big[F(j,j-m+1)[B(j-m)]^1,\cdots,[B(j-1)]^1,\cdots,\\ F(j,j-m+1)[B(j-m)]^p,\cdots,[B(j-1)]^p\big],
    \end{gathered}
\end{equation}
whose columns are the snapshots obtained through the $mp$ impulse-response simulations described above. Then 
\begin{equation}
X(j;m)X(j;m)^*=W_{ce}(j;m).
\label{Gramian_Controllability_Empirical Factor X}
\end{equation}
\label{Prop Empirical contrl gramian}
\end{prop}
Note that physically the empirical controllability Gramian is based on considering the total influence of the past history of impulse inputs on the `current state' $x(j)$ by neglecting those impulses $u(k)$, $k< j-m$, whose influence on the current state has mostly decayed with sufficiently large $m$.    

The above procedure is generally valid for calculation of the empirical controllability Gramian of a time-varying stable system.  The $T$-periodic feature of our system can substantially save the computational effort.  For instance, the snapshot $F(j,j-m+T+1)[B(j-m+T)]^d$ is obtained by running the `$(T+1)$-th' simulation mentioned above for $m-T$ steps, with initial condition $x(j-m+T)=0$ and corresponding unit impulse. However, to run this simulation is not necessary, since $F(j,j-m+T+1)[B(j-m+T)]^d=F(j-T,j-m+1)[B(j-m)]^d$, which has been calculated in the `first' simulation we mentioned above at the $(m-T)$-th step.  Indeed, it is clear that for a $T$-periodic system, when $m>T$, to obtain the $mp$ snapshots for construction of $X(j;m)$, one needs only run $Tp$ impulse-response simulations: For each column $[B(k)]^d$, starting from $x(j-m)=0$ run a simulation for $m$ steps. Then, start from $x(j-m+1)=0$, running for $m-1$ steps. Repeat the process until the simulation starting from $x(j-m+T-1)=0$ and running for $m-T+1$ steps.    The data set obtained is sufficient for construction of $X(j;m)$.

\subsubsection{Computation of the empirical observability Gramian}
\label{sec:obsv_gram}
To calculate $W_{oe}(j;m)$ for a fixed $j$, $1\leqslant j\leqslant T$, by the method of snapshots, one needs to construct an \emph{adjoint system}
\begin{equation}
z(k+1)=\hat{A}(k)z(k)+\hat{C}(k)v(k)
\label{adjoint system}
\end{equation}
where $k=j,\cdots,j+m-1$, the state $z(k)\in\mathbb{C}^n$, the control input $v(k)\in\mathbb{C}^q$, and 
\begin{align*}
\hat{A}(k)=A(2j+m-k-1)^*,\qquad \hat{C}(k)=C(2j+m-k-1)^*.
\end{align*}
As above,  here we run a series of impulse-response simulations for the adjoint system.  Consider the $d$-th column of $\hat{C}(k)$, $[\hat{C}(k)]^d=\left[C(2j+m-k-1)^*\right]^d$.  One starts the first simulation with the initial condition $z(j)=0$ and a unit impulse $v(j)=[0,\cdots,0,1,0,\cdots,0]^{\top}$ whose $d$-th entry is~$1$.  By running the simulation for $m$ steps, we obtain 
\begin{align*}
z(j+1)&=\left[C(j+m-1)^*\right]^d;\\
z(j+2)&=F(j+m-1,j+m-2)^*\left[C(j+m-1)^*\right]^d;\\
&\;\vdots\\
z(j+m)&=F(j+m-1,j)^*\left[C(j+m-1)^*\right]^d,
\end{align*}
where  $z(j+m)$ is a snapshot we need for computing the $W_{oe}(j;m)$.  The second simulation starts with initial condition $z(j+1)=0$ and unit impulse $v(j+1)=[0,\cdots,0,1,0,\cdots,0]^{\top}$, and by running $(m-1)$ steps we stop at $z(j+m)=F(j+m-2,j)^*\left[C(j+m-2)^*\right]^d$, another snapshot we need.  Repeat the above process until the $m$-th simulation, with initial condition $z(j+m-1)=0$ and unit impulse $v(j+m-1)=[0,\cdots,0,1,0,\cdots,0]^{\top}$, where only one step of simulation is needed to obtain $z(j+m)=\left[C(j)^*\right]^d$.   Since $J(k)$ has $q$ columns, we need to run $mq$ simulations, with $m(m+1)q/2$ steps total, to obtain $mq$ snapshots, with which we have the following statement: 

\begin{prop}{\bf (Observability Gramian from snapshots)}
Define an $n\times mq$ dimensional matrix
\begin{equation}
\begin{gathered}
    Y(j;m)=\Big[F(j+m-1,j)^*\left[C(j+m-1)^*\right]^1,\cdots,\left[C(j)^*\right]^1,\cdots,\\ F(j+m-1,j)^*\left[C(j+m-1)^*\right]^q,\cdots,\left[C(j)^*\right]^q\Big],
\end{gathered}
\end{equation}
whose columns are the $m_{oq}$ snapshots obtained through the $mq$ impulse-response simulations described above. Then 
\begin{equation}
Y(j;m)Y(j;m)^*=W_{oe}(j;m).
\label{Gramian_Observability_Empirical Factor Y}
\end{equation}
\label{Prop Empirical observb gramian}
\end{prop}
 
Again, the above procedure is valid for general stable time-varying systems.  If the system is $T$-periodic, then as for the empirical controllability Gramian case, one only needs to run $Tq$ simulations to obtain the $mq$ snapshots.

\subsection{Balanced truncation using the method of snapshots}
\label{sec:balpod}
Suppose we have obtained the factors $X(j;m_c)$ and $Y(j;m_o)$ mentioned above, where $m_c,m_o\in\mathbb{N}$ may be different. Then for the \emph{lifted LTI system at time $j$}, by Proposition \ref{Prop Gramian Contrl Observ}, \ref{Prop Empirical contrl gramian} and \ref{Prop Empirical observb gramian}, its Gramians are approximated by 
\begin{equation}
	\begin{aligned}
	\tilde{W}_{jc}&\approx W_{ce}(j;m_c)=X(j;m_c)X(j;m_c)^*\\
	\tilde{W}_{jo}&\approx W_{oe}(j;m_o)=Y(j;m_o)Y(j;m_o)^*.
	\end{aligned}
	\label{approximation of Gramians}
\end{equation}
For LTI systems, the method of snapshots presented in~\cite{Ro2005} gives an algorithm for computing the transformation that exactly balances the empirical Gramians $W_{ce}$ and $W_{oe}$, directly from the factors $X(j;m_c)$ and $Y(j;m_o)$, without computing the Gramians themselves:

\begin{thm}{\bf (Balanced truncation using the method of snapshots~\cite{Ro2005})}
Let $\Sigma\in\mathbb{C}^{a\times a}$ be a real diagonal matrix including the non-zero Hankel singular values, obtained by singular value decomposition (SVD) of the matrix $Y(j;m_o)^*X(j;m_c)$ 
\begin{equation}
Y(j;m_o)^*X(j;m_c)=U\Sigma V^*,
\label{SVD}
\end{equation}
 in which $a$ is the rank of $Y(j;m_o)^*X(j;m_c)$, and $U\in\mathbb{C}^{m_oq\times a}$, $V\in\mathbb{C}^{m_cp\times a}$ satisfy $U^*U=V^*V=I_{a\times a}$. The balancing transformation is then found by computing matrices 
\begin{equation}
\Phi=X(j;m_c)V\Sigma^{-1/2}; \quad\quad \Psi=Y(j;m_o)U\Sigma^{-1/2}.
\label{T1 S1}
\end{equation}
If $a=n$, then  $\Phi$ is the balancing transformation and $\Psi^*$ is its inverse.  If $a<n$, then the columns of $\Phi$ form the first $a$ columns of the balancing transformation, and the rows of $\Psi^*$ form the first $a$ rows of the inverse transformation.  
\end{thm}

We emphasize that computing the balancing transformation by the method of snapshots {\em exactly} balances the empirical Gramians.  The only approximation here is to use $X$ and $Y$ to approximately construct the Gramians as in (\ref{approximation of Gramians}). 

Balanced truncation of order $r$ is then done as follows.  Let $\Phi_{1}$ denote the first $r$ columns of $\Phi$, and $\Psi_{1}^*$ the first $r$ rows of $\Psi^*$. 
The reduced state variable $\tilde{z}_j(t)\in\mathbb{C}^r$ satisfies $\tilde{z}_j(t)=\Psi_{1}^*\tilde{x}_j(t)=\Psi_{1}^*x(j+tT)$.  The reduced model, in the lifted setting, is then given by
\begin{align}
\tilde{z}_j(t+1)&=\Psi_{1}^*\tilde{A}_j\Phi_{1}\tilde{z}_j(t)+\Psi_{1}^*\tilde{B}_j\tilde{u}_j(t);
\label{Reduced system part1}\\
\tilde{y}_j(t)&=\tilde{C}_j\Phi_{1}\tilde{z}_j(t)+\tilde{D}_j\tilde{u}_j(t),
\label{Reduced system part2}
\end{align}
In simulations, the reduced output equation (\ref{Reduced system part2}) shall be un-lifted to the original periodic setting: For each $i$, $i=1,\cdots T$,
\begin{equation}
\begin{aligned}
y(j+tT+i-1)&=C(j+i-1)F(j+i-1,j)\Phi_{1}\tilde{z}_j(t)\\
&\quad+\sum_{k=1}^{T}\tilde{D}_{j(i,k)}u(j+tT+k-1)
\end{aligned}
\label{Reduced system part2_periodic setting}
\end{equation}
where $\tilde{D}_{j(i,k)}$ denotes the entry of $\tilde{D}_j$ at $i$-th row and $k$-th column.

\subsection{Output projection method}
\label{sec:outputproj}
When the number of outputs $q$ is very large, direct construction of $Y(j;m_o)$ as described in Section~\ref{sec:obsv_gram} will be computationally intractable, because $Tq$ adjoint simulations are required. To overcome this, an \emph{output projection method}~\cite{Ro2005} may be used, by which one can substantially reduce the number of adjoint simulations. The starting point of this method is to define an optimization problem that minimizes the error between the input-output behavior of the original system and that of a projected system with a smaller-dimensional output space.  For a time-periodic system, this optimization problem is cumbersome to define directly.  We instead consider the lifted LTI system, for which we design a version of output projection, and then relate the method back to the periodic system.

Fix a time~$j$.  The lifted system at time~$j$, though with an even higher dimension of output~$Tq$, is a standard LTI system, and its input-output behavior can be measured by a sequence of $Tq\times Tp$ dimensional impulse-response matrices $\{\tilde{G}_j(t)\}$, the $i$-th column of each $\tilde{G}_j(t)$ representing  the output response $\tilde{y}_j(t)$ corresponding to a unit impulse input $\tilde{u}_j(0)=[0,\cdots,0,1,0\cdots,0]^{\top}$ whose $i$-th entry is $1$.   By output projection \cite{Ro2005}, we mean 
the search of an orthogonal projection on $\mathbb{C}^{Tq}$ with rank $\tilde{r}_{op}\ll Tq$, i.e. a $\tilde{P}_j=\tilde{\Theta}_j\tilde{\Theta}_j^*$ where $\tilde{\Theta}_j\in\mathbb{C}^{Tq\times \tilde{r}_{op}}$, $\tilde{\Theta}_j^*\tilde{\Theta}_j=I_{\tilde{r}_{op}\times \tilde{r}_{op}}$,  that leads to a \emph{projected lifted system at time $j$}
\begin{align}
\tilde{x}_j(t+1)&=\tilde{A}_j\tilde{x}_j(t)+\tilde{B}_j\tilde{u}_j(t);
\label{LTI part1 projected}\\
\tilde{y}_j(t)_{P}&=\tilde{P}_j\left(\tilde{C}_j\tilde{x}_j(t)+\tilde{D}_j\tilde{u}_j(t)\right),
\label{LTI part2 projected}
\end{align}
which is an approximation of the original lifted system, such that only $\tilde{r}_{op}$ adjoint simulations of the corresponding adjoint system are needed for calculation of the empirical observability Gramian of the projected system.  More details will be discussed soon. The $\tilde{P}_j$ shall be chosen such that the input-output behavior of the projected system is as close as possible to that of the original LTI system.   More precisely, $\tilde{P}_j$  is the solution of the optimization problem 
\begin{equation}
\min_{\substack{\{\tilde{P}\in \mathcal{P}_{\tilde{r}_{op}}\}}}\left(\sum_{t=0}^{\infty}||\tilde{G}_j(t)-\tilde{P}_j\tilde{G}_j(t)||^2\right)
\label{Error_outputProjection_lifted}
\end{equation}
with respect to some norm on matrices, where $\mathcal{P}_{r}$ denotes a space of rank-$r$ orthogonal projections.  If we use an induced norm, such as the Frobenius norm $||\cdot||_F$ induced by the inner product $\langle A, B\rangle=\operatorname{Trace}(A^*B)$, then the minimization problem has a standard solution: $\tilde{P}_j= \tilde{\Theta}_j \tilde{\Theta}_j^*$,  where $\tilde{\Theta}_j=\left[\tilde{\Theta}_j^1,\cdots,\tilde{\Theta}_j^{\tilde{r}_{op}}\right]$ in which the column vectors $\{\tilde{\Theta}_j^i\}$ are the orthonormal eigenvectors of $\mathcal{R}=\sum_{i=0}^{\infty}\tilde{G}_j(i)\tilde{G}_j(i)^*$.    
It is a typical eigenvalue problem in POD reduction and can be numerically solved by the method of snapshots \cite{Si1987}, where the snapshots are provided by the data sets $\{\tilde{G}_j(i)\}_{i=0}^{s}$ obtained through simulations.   

The above output projection generally produces a full matrix $\tilde{P}_j$. Thus the components of the projected $\tilde{y}_j(t)_P=\tilde{P}_j\tilde{y}_j(t)$ no longer cleanly correspond to the outputs of the periodic system at different time steps respectively, but the intermixed combinations of them, which is not desirable both for physical understanding of the system, and for numerical simulation purposes (recall that in simulations we do not really compute with the lifted system, so the projected lifted system should be `unlifted' back to a periodic system for computation).   To deal with this difficulty, we propose a modified version of output projection, in which we seek a sub-optimal solution that solves (\ref{Error_outputProjection_lifted}) with a constraint imposing that the orthogonal projection $\tilde{P}_j$ with rank $\tilde{r}_{op}$ takes a diagonal form      
\begin{equation}
\tilde{P}_j=\operatorname{diag}\left[\tilde{P}_j(1),\cdots,\tilde{P}_j(T)\right]
\label{P_Rtilde constrained}
\end{equation}
where each $q\times q$ block  is an orthogonal projection on $\mathbb{C}^q$ with rank $r_{op}$ in form of $\tilde{P}_j(i)=\tilde{\Theta}_j(i)\tilde{\Theta}_j(i)^*$ where $\tilde{\Theta}_j(i)\in\mathbb{C}^{q\times r_{op}}$ and $\tilde{\Theta}_j(i)^*\tilde{\Theta}_j(i)=I_{r_{op}\times r_{op}}$.  Note that in this case
\begin{equation}
\tilde{\Theta}_j=\operatorname{diag}\left[\tilde{\Theta}_j(1),\cdots,\tilde{\Theta}_j(T)\right].
\label{Phi_Rtilde constrained}
\end{equation}
Here we need $\tilde{r}_{op}=r_{op}T$.

The diagonal $\tilde{P}_j$ makes it easy to unlift the projected lifted system (\ref{LTI part1 projected}) \& (\ref{LTI part2 projected}) to a \emph{projected time-periodic system  }
\begin{align}
x(k+1)&=A(k)x(k)+B(k)u(k);
\label{LPS projected part1}
\\
y(k)_P&=P(k)C(k)x(k)
\label{LPS projected part2}
\end{align}
for $k=j,\cdots$, where the $T$-periodic orthogonal projection $P$ of rank $r_{op}$ on $\mathbb{C}^q$ is defined by 
\begin{equation}
\begin{aligned}
P(j+tT+i)=P(j+i):=\tilde{P}_j(i+1), \quad &i=0,\cdots,T-1,\\ 
&w=0,1,\cdots.
\end{aligned}
\label{P def}
\end{equation}
We write $P(k)=\Theta(k)\Theta(k)^*$ where the $T$-periodic $\Theta\in\mathbb{C}^{q\times r_{op}}$ is defined by $\Theta(j+tT+i)=\Theta(j+i):=\tilde{\Theta}_j(i+1)$. 

By the above unlifting, though solving for the sub-optimal $\tilde{P}_j$ is no longer a standard POD reduction problem in the lifted setting, we can attack it in the periodic setting such that standard POD reduction techniques are available.  

To realize that, first we rewrite $\tilde{G}_j(t)$ as
\begin{equation}
\tilde{G}_j(t)=\begin{bmatrix}
\tilde{G}_j(t)_1\\
\vdots\\
\tilde{G}_j(t)_T
\end{bmatrix}
\label{Gtilde_blocks}
\end{equation}
where each block $\tilde{G}_j(t)_i$ is a $q\times Tp$ matrix.  Let $0\leqslant b\leqslant (T-1)$, $1\leqslant c\leqslant p$.  The $(bp+c)$-th column of $\tilde{G}_j(t)_i$ represents the output response of the original periodic system at time $(j+tT+i-1)$ to the unit impulse $u(j+b)=[0,\cdots,0,1,0,\cdots, 0]^{\top}$ whose $c$-th component is~$1$.  Thus, for the periodic system (\ref{LPS}), we define 
\begin{equation}
G(j+tT+i,j):=\tilde{G}_j(t)_{i+1}
\label{G matrix def}
\end{equation}
as its impulse-response matrix at time $(j+tT+i)$, since it includes all different responses at the current time respectively to corresponding unit impulse inputs during the whole time period $[j,(j+T-1)]$.  This definition matches that proposed in Bamieh and Pearson 1992 \cite{BaPe1992}. 

The following statement links the constrained optimization problem given above for the lifted system to an equivalent optimization problem defined in the periodic system.

\begin{prop}
Under the Frobenius norm, the optimization problem (\ref{Error_outputProjection_lifted}), with its solution $\tilde{P}_j$ constrained in form of (\ref{P_Rtilde constrained}), is equivalent to  the optimization problem
\begin{equation}
\begin{aligned}
\min_{\substack{\{P(j+i){\in \mathcal{P}_{r_{op}}},\\i=0,\cdots,T-1\}}}
&\left(\sum_{i=0}^{T-1}\sum_{t=0}^{\infty}\Big\|G(j+tT+i,j)\right.\\
&\left.-P(j+i)G(j+tT+i,j)\Big\|^2\right),
\label{Error_outputProjection_periodic}
\end{aligned}
\end{equation}
which implies, for each $i=0,\cdots,T-1$, $P(j+i)$ is the solution of the optimization problem  
\begin{equation}
\begin{aligned}
\min_{\substack{\{P(j+i){\in \mathcal{P}_{r_{op}}}\}}}
&\left(\sum_{t=0}^{\infty}\Big\|G(j+tT+i,j)\right.\\
&\left.-P(j+i)G(j+tT+i,j)\Big\|^2\right).
\label{Error_outputProjection_periodic unconstrained}
\end{aligned}
\end{equation}
\label{Prop Optimization}
\end{prop}
\begin{proof}
By direct calculation, using (\ref{P_Rtilde constrained}), (\ref{Gtilde_blocks}), (\ref{P def}), (\ref{G matrix def}) and the linearity of the $\operatorname{Trace}$ operation.
\end{proof}

This statement allows us to obtain the $\tilde{P}_j$ in form of~(\ref{P_Rtilde constrained}) by solving $T$ unconstrained optimization problems for $P(k)=\Theta(k)\Theta(k)^*$,  $k=j,\cdots,j+T-1$, in the periodic setting, each of which is a typical eigenvalue problem in POD reduction with solutions satisfying that the $r_{op}$ columns of each $\Theta(k)$, $\{\theta(k)^l\}_{l=1}^{r_{op}}$, are orthonormal  eigenvectors of $\mathcal{R}(k)=\sum_{t=0}^{\infty}G(tT+k,j)G(tT+k,j)^*$, i.e., they satisfy
\begin{equation} 
R(k)\theta(k)^{l}=\lambda_l\theta(k)^{l}.
\label{eigenvalue problem POD}
\end{equation}

Numerically, the eigenvectors can be solved by the method of snapshots, in which $T$ SVDs are needed to obtain $\Theta(j)$, $\cdots$, $\Theta(j+T-1)$.  See details of this method in, for example,  \cite{Si1987} and \cite{Ro2005}.  The snapshots are the columns of impulse-response matrices $\{G(j+tT+i,j)\}_{t=0}^{s}$ for $i=0,\cdots,T-1$.

A convenient computational feature of the output projection method is that by periodicity we see all the snapshots have already been obtained during the computation of $X(j;m_c)$ described in Section~\ref{sec:ctrb_gram}, as long as $m_c\geqslant (s+1)T$, except that the data obtained there needs to be left-multiplied with a corresponding $C$ matrix.  For instance, the matrix $C(j)X(j;m_c)$ includes all the columns of matrices $G(j+tT,j)$ for $w=1,\cdots, m_c/T$.

Suppose we have found the $\tilde{P}_j$ in form of (\ref{P_Rtilde constrained}).  By definition, the observability Gramian of the projected lifted system at time $j$ is
\begin{equation}
\tilde{W}_{joP}=\sum_{i=0}^{\infty}\left(\tilde{A}_j^i\right)^*\tilde{C}_j^*\tilde{\Theta}_j\tilde{\Theta}_j^*\tilde{C}_j\tilde{A}_j^i,
\label{Gramian_Observability_Projected_liftedSystem}  
\end{equation}
and the observability Gramian at time $j$ for the projected periodic system is 
\begin{equation}
W_{oP}(j)=\sum_{i=j}^{\infty}F(i,j)^*C(i)^*\Theta(i)\Theta(i)^*C(i)F(i,j).
\label{Gramian_Observability_Projected periodciSystem}  
\end{equation}

The following statement for the projected systems is an analog to Proposition \ref{Prop Gramian Contrl Observ} for the full case. Its proof is by direct calculation. 

\begin{prop}
The observability Gramian for the projected lifted system at time $j$, $\tilde{W}_{joP}$, is equal to the observability Gramian at time $j$ for the projected periodic system, $W_{oP}(j)$. 
\label{Prop Gramian Observ projected}
\end{prop}

Therefore, $\tilde{W}_{joP} = W_{oP}(j)\approx W_{oPe}(j;m_o)$, where $W_{oPe}(j;m_o)$ is the empirical observability Gramian at time $j$ of the projected periodic system obtained by the method of snapshots introduced in  Section 3.1.  Now  the input is only $r_{op}$ dimensional, $r_{op}\ll q$, in the corresponding \emph{projected adjoint time-periodic system}
\begin{equation}
z(k+1)=\hat{A}(k)z(k)+\hat{C}_P(k)v_{r_{op}}(k)
\label{adjoint system_projected}
\end{equation}
where $k=j,\cdots,j+m_o-1$, the state $z(k)\in\mathbb{C}^n$, the control input $v_{r_{op}}(k)\in\mathbb{C}^{r_{op}}$, and 
\begin{align*}
\hat{A}(k)&=A(2j+m_o-k-1)^*,\\
\hat{C}_P(k)&=C(2j+m_o-k-1)^*\Theta(2j+m_o-k-1).
\end{align*}
Thus, only $Tr_{op}$ adjoint simulations in total is needed for computing $W_{eoP}(j;m_o)$.  

\begin{rem}
{\rm Instead of seeking sub-optimal solutions for (\ref{Error_outputProjection_lifted}) in form of (\ref{P_Rtilde constrained}), (\ref{Phi_Rtilde constrained}), where each $\tilde{\Theta}_j(i)=\Theta(j+i-1)$, $i=1,\cdots, T$, are different and given by the solutions of (\ref{eigenvalue problem POD}), one can also alternatively impose a stronger constraint that all the projection matrices at each time step are the same, which means  $\Theta(j)=\Theta(j+1)=\cdots=\Theta(j+T-1):=\Theta$, and   
\begin{equation}
\tilde{\Theta}_j=\operatorname{diag}\left[\Theta,\cdots,\Theta\right].
\label{Phi_Rtilde constrained_one Phi}
\end{equation}
It is straightforward to show that, similar to (\ref{eigenvalue problem POD}), the single $\Theta$ is obtained by solving  the eigenvalue problem
\begin{equation}
R\theta^{l}=\lambda_l\theta^{l},
\label{eigenvalue problem POD one Phi}
\end{equation}
where $\{\theta^l\}_{l=1}^{r_{op}}$ are the columns of $\Theta$, and $R=\sum_{k=j}^{j+T-1}\sum_{t=0}^{\infty}G(tT+k,j)G(tT+k,j)^*$.  The $\Theta$ obtained by this setting reflects the overall impulse-responses of the periodic system, mixing those at different time steps along each time period.  The stronger constraint given here let us expect that the solution is even less optimal.  However, a numerical advantage is that only one SVD is needed in the method of snapshots for solving the single $\Theta$. 
}
\label{remark single projection}
\end{rem}
\vskip10pt

\subsection{Summary: procedures of balanced POD }
Following the terminology in \cite{Ro2005}, the approximate balanced truncation based on the method of snapshots given in Section \ref{snapshot_subsec} \& \ref{sec:balpod}, plus the output projection introduced in Section~\ref{sec:outputproj},  is named \emph{balanced POD}.  We summarize the computational procedures of the balanced POD method for linear asymptotically stable time-periodic systems:
\begin{itemize}
\item{Step 0: Pick a time $j$, $1\leqslant j\leqslant T$, as the ``base point,'' based on which the balanced truncation will be done.}
\item{Step 1: Run $Tp$ impulse-response simulations to obtain $m_c p$ snapshots and form $n\times m_c p$ dimensional $X(j;m_c)$ as described in Section~\ref{snapshot_subsec}.}
\item{Step 2: For each $(j+i)$, $i=0,\cdots,T-1$ (corresponding to one whole period), left-multiply the state responses stored during computing $X(j;m_c)$ by corresponding $C$ matrices, and then use the method of snapshots to solve the $T$ eigenvalue problems defined by (\ref{eigenvalue problem POD}) for the $T$ output projection matrices $P(j+i)$ along one time period.}
\item{Step 3: Construct the ``projected adjoint system''~(\ref{adjoint system_projected}) and run $Tr_{op}$ impulse-response simulations for the adjoint system to form $n\times m_o r_{op}$ dimensional $Y(j;m_o)$, as described in Section~\ref{snapshot_subsec}}
\item{Step 4: Compute the SVD of $Y(j;m_o)^*X(j;m_c)$ defined in (\ref{SVD}) and compute the balancing POD modes for the lifted system given by (\ref{T1 S1}).}
\item{Step 5: Obtain the reduced lifted system in form of (\ref{Reduced system part1}) \&~(\ref{Reduced system part2}). For computational purposes, rewrite the output equation (\ref{Reduced system part2}) as (\ref{Reduced system part2_periodic setting}). }
\end{itemize}

If the dimension of output $q$ is small, then we can skip Step 2 for output projection, and directly run $Tq$ impulse-response simulations for the adjoint system (\ref{adjoint system}) as in Section 3.1 to form the $n\times m_o q$ dimensional $Y(j;m_o)$.    On the other hand, in Step 2, instead of solving for the time varying $T$-periodic output projection matrices,  one can alternatively seek a single time-invariant projection matrix by solving the eigenvalue problem (\ref{eigenvalue problem POD one Phi}) as we mentioned in Remark \ref{remark single projection}. 

\begin{rem}
{\rm Note that to pick up an optimal base time $j$ in the sense that the upper error bound of the reduced system of a fixed order is minimized,  one needs to solve for the Hankel singular value matrices $\Sigma$ corresponding to the lifted systems at different times during one whole period (see Farhood et al.~\cite{FaBeDu2005}). This method is therefore not attractive or even computationally feasible for large systems.} 
\label{remark optimal base time}
\end{rem}

\begin{rem}
{\rm 
A ``dual'' version of the above algorithm is readily available for balanced truncations of linear periodic systems with high-dimensional states and inputs, but only few outputs.  In that case, one can start with the construction of $Y(j;m_o)$ by running a series of impulse-response simulations for the adjoint system, since the dimension of outputs is small.  Then an \emph{input projection} based on POD reduction shall be done in the same spirit as that underlies the output projection to obtain a projected system whose dimension of inputs is reasonably small. One then runs a corresponding number of impulse-response simulations for the projected system to construct $X(j;m_c)$.
}
\label{remark input projection}
\end{rem}

\section{Example}
\label{example_sec}

To validate and demonstrate the balanced POD algorithm, we consider the following example.  Consider a linear periodic system~(\ref{LPS}) with period $T=5$, state dimension $n=30$, output dimension $q=30$, control input dimension $p=1$, and $\{A(k)\}_{k=1}^{5}$ are randomly generated diagonal matrices with diagonal entries bounded in $[0.16, 0.96]$, guaranteeing the asymptotical stability.  $\{B(k)\}$ and $\{C(k)\}$ matrices are also randomly generated, with entries bounded in $[0,1]$.  The setting is similar to that used in the example in Farhood, et. al \cite{FaBeDu2005}.   It is not a high-dimensional problem, such that exact balanced realizations for the corresponding lifted system can be done by solving Lyapunov equations for Gramians of the lifted system, and the result can thus be compared with that by the balanced POD approach. 

We pick the ``base time'' $j=1$.  Choose $m_c=m_o=2T=10$, and test different cases of balanced truncation, including balanced PODs with the order of output projection for the periodic system at each time step $r_{op}=1, 2, 6$ and $10$ respectively. Recall that, to the lifted system, the order of output projection is  $\tilde{r}_{op}=r_{op}T$.  Figure \ref{Figure Hankel Singular Values} shows the Hankel singular values in the lifted setting, obtained by the exact balanced realization, the balanced truncation based on the method of snapshots given in Section~\ref{sec:balpod} but without output projection, and the balanced POD (with $T$ different output projection matrices along one period for the periodic system).  Figure \ref{Figure G inf norm} shows the error plots of the infinity norm, $||\tilde{G}-\tilde{G}_{\tilde{r}}||_\infty/||\tilde{G}||_\infty$ versus $\tilde{r}$, for those different cases. Here $\tilde{G}_{\tilde{r}}$ is the impulse-response matrix of the reduced lifted system of order $\tilde{r}$.  We do not show the balanced POD with $r_{op}=10$ case since the result is almost identical to that by balanced truncation based on the method of snapshots but without output projection. We see that the balanced truncation based on the method of snapshots gives a good approximation of the exact balanced truncation, and further, the balanced POD, even with low orders of output projection $r_{op}$, generates satisfying results.  Figure \ref{Figure G inf norm diff output projections} shows comparisons between balanced POD results with the same order of output projection, one set based on $T$ different projection matrices along one period in the periodic setting, and the other only having one time-invariant projection matrix (see Section~\ref{sec:outputproj}).   For the cases where $r_{op}$ are low, these two approaches give almost identical results, or even the latter one gives better results.  However, when the order of output projection $r_{op}$ increases, such as $r_{op}=6$, the results based on $T$ different projection matrices at each time step are better than those by a single projection matrix, as we expect.

\begin{figure}[htb]
\centering
\includegraphics[width=0.6\linewidth]{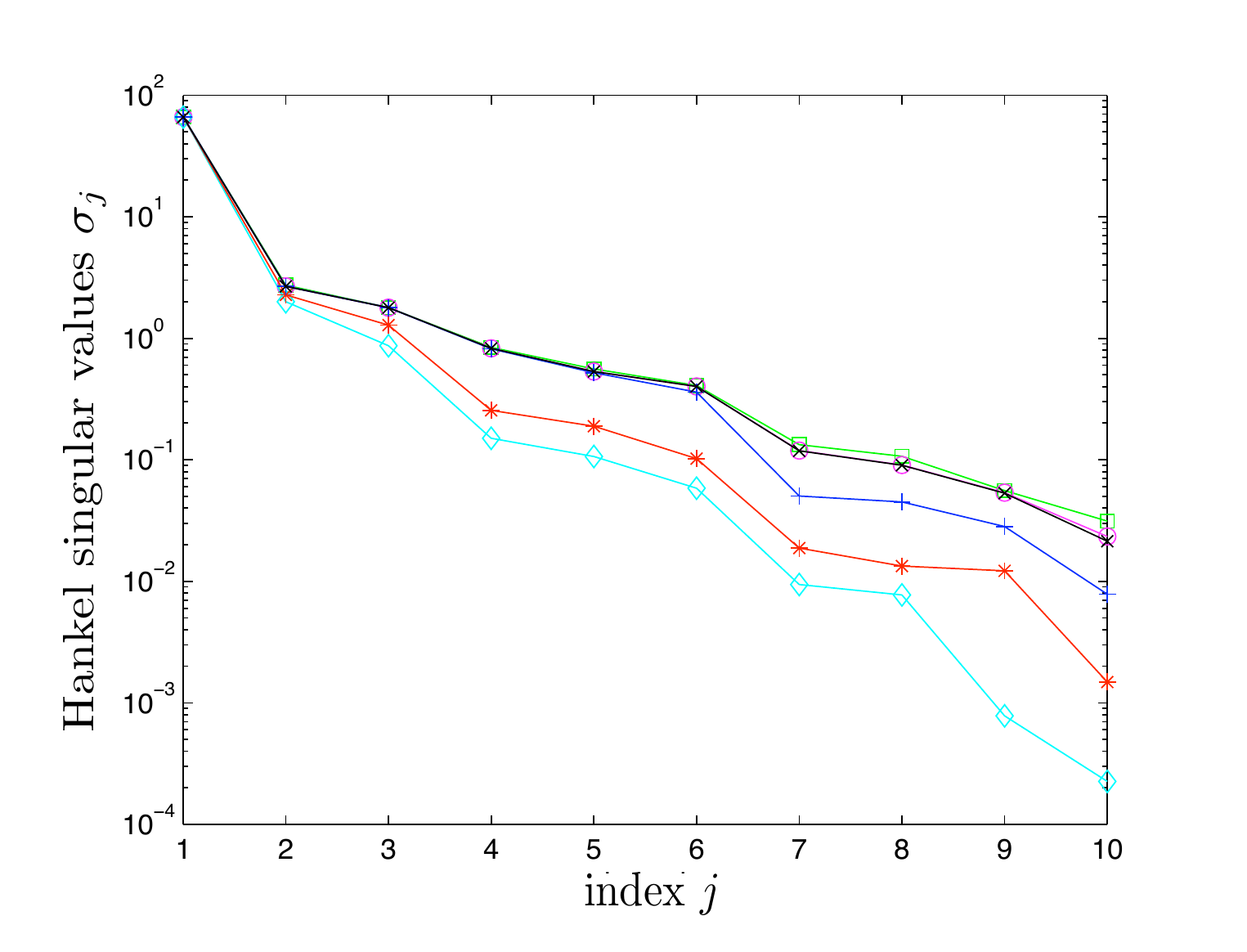}
\caption{The Hankel singular values $\sigma_j$: exact balanced truncation(\textcolor{green}{$\square$}); balanced truncation by the method of snapshots but without output projection(\textcolor{magenta}{$\circ$}); balanced POD with $r_{op}=1$ (\textcolor{cyan}{$\lozenge$}); balanced POD with $r_{op}=2$ (\textcolor{red}{$*$}); balanced POD with $r_{op}=6$ (\textcolor{blue}{$+$}); balanced POD with $r_{op}=10$ (\textcolor{black}{$\times$}).   }
\label{Figure Hankel Singular Values}
\end{figure}

\begin{figure}[htb]
\centering
\includegraphics[width=0.6\linewidth]{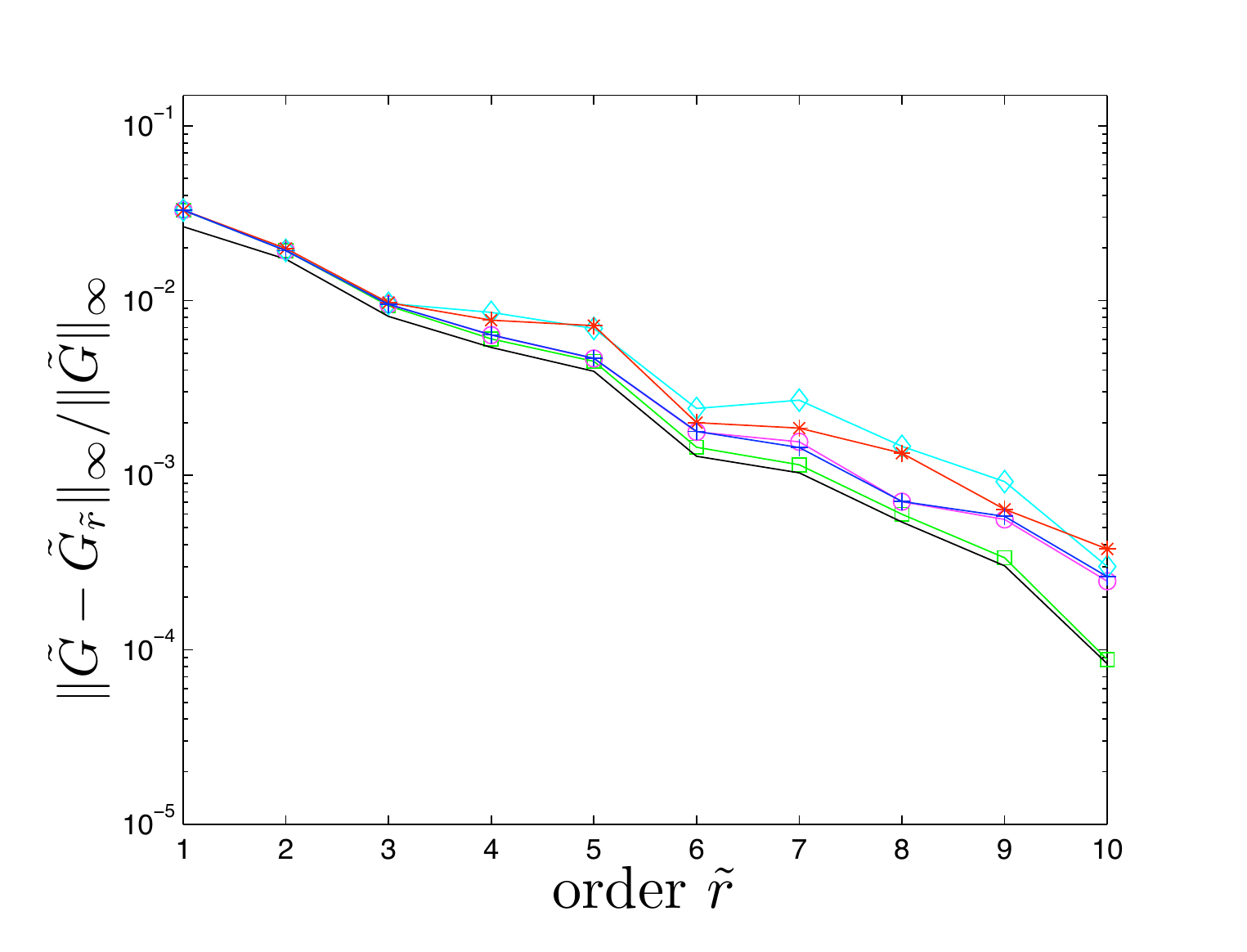}
\caption{Error $||\tilde{G}-\tilde{G}_{\tilde{r}}||_\infty/||\tilde{G}||_\infty$, for exact balanced truncation(\textcolor{green}{$\square$}), balanced truncation by the method of snapshots but without output projection(\textcolor{magenta}{$\circ$}), balanced POD with $r_{op}=1$ (\textcolor{cyan}{$\lozenge$}), balanced POD with $r_{op}=2$ (\textcolor{red}{$*$}), balanced POD with $r_{op}=6$ (\textcolor{blue}{$+$}), and the lower bound for any model reduction scheme (\textcolor{black}{$-$}).}
\label{Figure G inf norm}
\end{figure}

\begin{figure}[htb]
\centering
\includegraphics[width=0.6\linewidth]{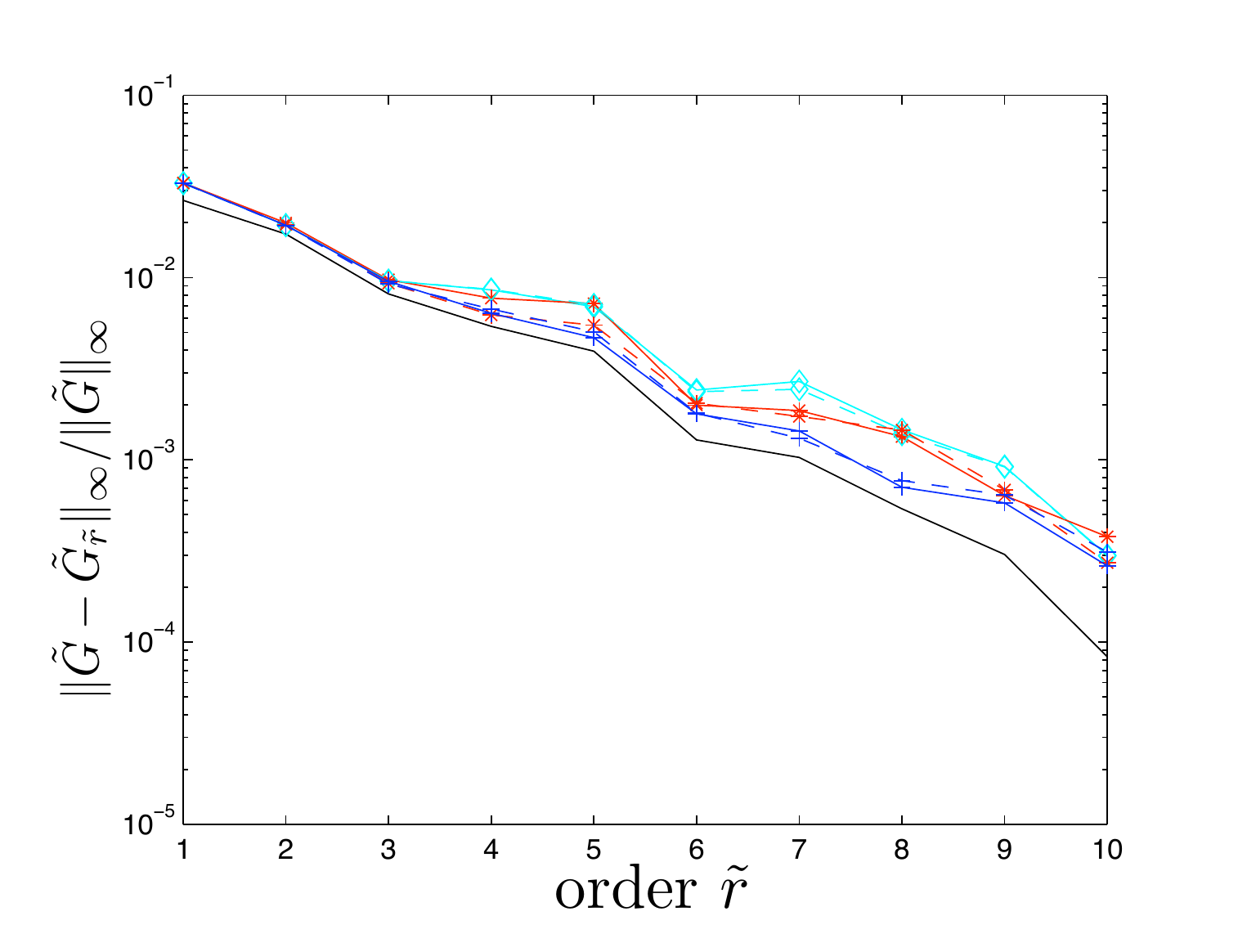}
\caption{Time varying $T$-periodic output projections versus time-invariant output projections: Error $||\tilde{G}-\tilde{G}_{\tilde{r}}||_\infty/||\tilde{G}||_\infty$, for balanced POD with $r_{op}=1$ (\textcolor{cyan}{$\lozenge$}), balanced POD with $r_{op}=2$ (\textcolor{red}{$*$}) and balanced POD with $r_{op}=6$ (\textcolor{blue}{$+$}). Solid lines correspond to cases using $T$ different projection matrices along one period for the periodic system, and dashed lines using one single projection matrix.  The black solid line is the lower bound for any model reduction scheme (\textcolor{black}{$-$}).}
\label{Figure G inf norm diff output projections}
\end{figure}

\section{Conclusion and future directions}
We have proposed a version of the balanced POD method to realize approximate balanced truncation for linear asymptotically stable periodic systems, especially with very high-dimensional states and outputs but a small number of inputs. It is a generalization of the balanced POD method for linear time invariant system developed in~\cite{Ro2005}.  The development of this balanced POD is based on a lifting approach, and key parts include the method of snapshots for computing empirical Gramians, and a version of output projection that gives a periodic projected system based on POD reduction, with which the number of adjoint simulations needed for computing the empirical observability Gramian is substantially decreased.  Simulation results given in the previous section validate the approach. This snapshot-based approach is also readily applicable to high-dimensional systems with  high-dimensional inputs and few outputs, where an input projection is needed. 

A future direction of this work is to apply the balanced POD method to construct reduced-order models of high-dimensional (linearized) periodic systems arising in engineering applications and then, based on the low-dimensional models, to design closed-loop control laws based on these models.  For instance, such periodic orbits may arise as periodic shedding in the wake of a bluff body~\cite{Noack-03}, or from open-loop forcing at a prescribed frequency, as in the recent results of~\cite{MiKaSpKim2006}, which show that a periodic blowing and suction of flow added at the walls of a channel flow may reduce drag. The dimension of states of such systems, including the three components of velocity and pressure at each grid point in the channel, can be on the order of $10^6$, and thus reduced-order models of such systems are quite valuable for designing model-based control laws.     

\begin{appendix}

\begin{paragraph}{Observability Gramian for linear time-varying systems}

The observability Gramian provides a measure of the influence of an initial state $x(j)$ on future outputs with zero control inputs.  To see that, first, for a fixed time $j$, we define a linear operator $\Psi_{oj}: \mathbb{C}^n\rightarrow l^2[j,\infty)$ for system (\ref{LPS}) to describe the state-output behavior $y=\Psi_{oj}x(j)$ with zero inputs and an initial state $x(j)$.  More precisely, $y(j+k)=C(j+k)F(j+k,j)x(j)$, $k=0,1,\cdots$. To measure to what degree the state $x(j)$ excites the output $y$, it is natural to compute the square of the induced norm 
\begin{align*}
||y||_{l^2}^2=\langle y,y\rangle_{l^2}
\end{align*}
where $\langle\cdot,\cdot\rangle$ denotes an inner product, with subscripts specifying the vector space where the inner product is defined.  The observability Gramian given in (\ref{Grammian_def}) has the following property:  
\begin{prop}
\begin{equation*}
W_o(j)=\Psi_{oj}^*\Psi_{oj},
\end{equation*}
where $\Psi_{oj}^*: l^2[j,\infty)\rightarrow\mathbb{C}^n $ is the adjoint of $\Psi_{oj}$.  
And, 
\begin{align*}
||y||_{l^2}^2=\left\langle x(j),W_{o}(j)x(j)\right\rangle_{\mathbb{C}^n}.
\end{align*}
\label{Prop Observability Gramian property}
\end{prop}
\begin{proof}
First, it is clear that 
\begin{align*}
||y||_{l^2}^2=\langle y,y\rangle_{l^2}=\left\langle\Psi_{oj}x(j),\Psi_{oj}x(j)\right\rangle_{l^2}=\left\langle x(j),\Psi_{oj}^*\Psi_{oj}x(j)\right\rangle_{\mathbb{C}^n}.
\end{align*}

To explicitly express $\Psi_{oj}^*$, we consider, for an arbitrary $z\in l^2[j,\infty)$, 
\begin{align*}
\left\langle\Psi_{oj}^*z,x(j)\right\rangle_{\mathbb{C}^n}&=\left\langle z,\Psi_{oj}x(j)\right\rangle_{l^2}\\
&=\sum_{i=j}^{\infty}\left\langle z(i),C(i)F(i,j)x(j)\right\rangle_{\mathbb{C}^n}\\
&=\left\langle \sum_{i=j}^{\infty} F(i,j)^*C(i)^*z(i),x(j)\right\rangle_{\mathbb{C}^n}
\end{align*}
where matrices are considered as linear operators and $(\cdot)^*$ denotes the corresponding adjoint. So $\Psi_{oj}^*z=\sum_{i=j}^\infty F(i,j)^*C(i)^*z(i)$, and 
$\Psi_{oj}^*(\Psi_{oj}x(j))=\sum_{i=j}^\infty F(i,j)^*C(i)^*C(i)F(i,j)x(j)=W_{o}(j)x(j)$.  
\end{proof}

Note that the observability Gramian has non-negative eigenvalues, the larger ones corresponding to the more observable states.  
\end{paragraph}

\begin{paragraph}{Controllability Gramian for linear time-varying systems}
Similarly, the controllability Gramian provides a measure of the influence of a sequence of input history on the current state.  In other words, to reach a given current state (if possible), the controllability Gramian measures how much effort from control inputs is needed.  For a fixed current time $j$, define a linear operator $\Psi_{cj}:l^2(-\infty,j-1]\rightarrow \mathbb{C}^n$ such that $x(j)=\Psi_{cj}u=\sum_{i=-\infty}^{j-1}F(j,i+1)B(i)u(i)$. This operator describes the input-state behavior with initial state $x(-\infty)=0$ and a sequence of inputs $\{u(i)\}_{-\infty}^{j-1}$.  Consider the `energy' of input needed for reaching the current state $x(j)$ (suppose the system is controllable at time $j$)
\begin{align*}
||u||_{l^2}^2=\langle u,u\rangle_{l^2}.
\end{align*}
We have: 
\begin{prop} The controllability Gramian $W_{c}(j)$ defined in (\ref{Grammian_def}) satisfies
\begin{equation*}
W_c(j)=\Psi_{cj}\Psi_{cj}^*,
\end{equation*}
where $\Psi_{cj}^*:\mathbb{C}^n\rightarrow l^2(-\infty,j-1]$ is the adjoint of $\Psi_{cj}$. 
And, 
\begin{align*}
||u||_{l^2}^2=\left\langle x(j),W_{c}(j)^{-1}x(j)\right\rangle_{\mathbb{C}^n}
\end{align*}
where $(\cdot)^{-1}$ is the inverse of $(\cdot)$.
\label{Prop Controllability Gramian property}
\end{prop}
\begin{proof}
First,
\begin{align*}
||u||_{l^2}^2=\langle u,u\rangle_{l^2}&=\left\langle\Psi_{cj}^{-1}x(j),\Psi_{cj}^{-1}x(j)\right\rangle_{l^2}\\
&=\left\langle x(j),\left(\Psi_{cj}\Psi_{cj}^*\right)^{-1}x(j)\right\rangle_{\mathbb{C}^n},
\end{align*}
where one uses the fact that $\left(\Psi_{cj}^{-1}\right)^*=\left(\Psi_{cj}^{*}\right)^{-1}$. 

Similar to the observability Gramian case, by calculations under standard inner products of $\mathbb{C}^n$ and $l^2(-\infty,j-1]$, one obtains 
$\left(\Psi_{cj}^*z\right)(i)=B(i)^*F(j,i+1)^*z$, $i=-\infty,\cdots,j-1$, $z\in\mathbb{C}^n$. It follows from definition that $W_c(j)=\Psi_{cj}\Psi_{cj}^*$.  
\end{proof}

Note that the eigenvalues of the controllability Gramian are non-negative and the larger ones correspond to the more controllable states. 
 
\end{paragraph}

\end{appendix}

\footnotesize{
}

\end{document}